\documentclass[11pt,twoside, final]{amsart}
\copyrightinfo{0}{Iranian Mathematical Society}
\pagespan{1}{\pageref*{LastPage}}
\usepackage{etoolbox,lastpage}
\commby{}
\date{\scriptsize   Received: , Accepted: .}
\usepackage{amsmath,amsthm,amscd,amsfonts,amssymb,enumerate}
\usepackage{graphicx}	
 \usepackage{pb-diagram} 	
\usepackage{color}
\usepackage[colorlinks]{hyperref}

\newtheorem{theorem}{Theorem}[section]
\newtheorem{proposition}[theorem]{Proposition}
\newtheorem{lemma}[theorem]{Lemma}

\theoremstyle{definition}
\newtheorem{definition}[theorem]{Definition}

\theoremstyle{remark}

 \newtheorem{claim}[theorem]{Claim}
\numberwithin{equation}{section}
 \newcommand{\PP}{{\mathbb{P}}}

 \DeclareMathOperator{\dom}{dom}

 \DeclareMathOperator{\ZFC}{ZFC}

 \def\k{\kappa}

 \def\rmark{\mbox{$\rm\bf\rule{0.06em}{1.45ex}\kern-0.05em R$}}
 \def\pmark{\mbox{$\rm\bf\rule{0.06em}{1.45ex}\kern-0.05em P$}}
 \def\nmark{\mbox{$\rm\bf\rule{0.06em}{1.45ex}\kern-0.05em N$}}
 \def\vdash{\mbox{$\rm\| \kern-0.13em  - $}}
 \newcommand{\lusim}[1]{\smash{\underset{\raisebox{1.2pt}[0cm][0cm]{$\sim$}}
 {{#1}}}}

\begin{document}

%% The title of the paper goes here.  Edit your title.

\title[On Shelah Cardinals]{On Some Properties of Shelah Cardinals}
%% Now edit the following to give First Author name and address:
%% $^*$ for the corresponding author.

\author[A. S. Daghighi]{Ali Sadegh Daghighi}
\address[Ali Sadegh Daghighi]{Department of Mathematics and Computer Science \\
              Amirkabir University of Technology\\
              Hafez avenue 15194, Tehran, Iran\\  }
\email{a.s.daghighi@gmail.com}

\author[M. Pourmahdian]{Massoud Pourmahdian $^*$}
\address[Massoud Pourmahdian]{Department of Mathematics and Computer Science \\
              Amirkabir University of Technology\\
              Hafez avenue 15194, Tehran, Iran\\ }
\email{pourmahd@ipm.ir}
%% If there are three of more authors they are added in the obvious way.

  \thanks{$^*$Corresponding author}
%------------------------------------------------------------------------------------%
%%
%% Use the following command to make the title for the paper.
%
 %\CoverPage

 \maketitle
%
%%% The following environment is needed for the abstract.
%%%

\begin{abstract}
We present several results concerning Shelah cardinals including the fact that small and fast function forcings preserve Shelah and $(^{\kappa}\kappa\cap V)$-Shelah cardinals respectively. Furthermore we prove that the Laver Diamond Principle holds for Shelah cardinals and use this fact to show that Shelah cardinals can be made indestructible under $\leq \kappa$-directed closed forcings of size $<wt(\kappa)$.
\end{abstract}

\section{\bf Introduction}

Shelah cardinals were originally introduced by Shelah \cite{shelah-woodin} in connection with some problems in measure theory and descriptive set theory. An uncountable cardinal $\k$ is called Shelah, if for every function $f:\k \rightarrow \k,$ there exists an elementary embedding $j :  V \rightarrow M$ with $crit(j)=\k$ such that $^{\k}M \subseteq M$ and $V_{j(f)(\k)} \subseteq M$. Later Gitik and Shelah \cite{Gitik-Shelah} introduced the generalized notion of an $A$-Shelah cardinal for a set $A\subseteq$ $^{\kappa}\kappa$ (see definition \ref{A-shelah cardinal}). Shelah cardinals lie between Woodin and supercompact cardinals in the consistency strength hierarchy while the consistency strength of $A$-Shelah cardinals varies depending on $A$.

Investigating preservation of large cardinals under various classes of forcing notions started by the work of Levy and Solovay \cite{levy-solovay}, who showed that measurable cardinals are preserved under small forcing notions.  Later works (see for example \cite{jech},  \cite{hamkins1},  \cite{hamkins2}, \cite{hamkins-woodin}) revealed the fact that besides measurable cardinals, the same result holds for a wide range of large cardinals as well. In section 2, we prove an analogue of the Levy-Solovay theorem for Shelah cardinals, by showing that Shelah cardinals are preserved under small forcing notions. In section 3, we also prove that $(^{\kappa}\kappa\cap V)$-Shelah cardinals are preserved by Woodin's fast function forcing.

In section 4, we deal with the Laver Diamond Principle, isolated by Hamkins \cite{hamkins-diamond} as a new combinatorial axiom generalizing the classical Diamond Principle, and prove that such a principle holds for Shelah cardinals. Similar results are already obtained by Laver \cite{laver} and Gitik and Shelah \cite{Gitik-Shelah} for the other large cardinals including supercompact and strong cardinals respectively. Then in section 5 we use the already obtained Laver function for Shelah cardinals to prove an analogue of Laver's indestructibility of supercompactness theorem \cite{laver} for Shelah cardinals by proving that Shelah cardinals can be made indestructible under $\leq \kappa$-directed closed forcings of size $<wt(\kappa)$.    

The present work could be viewed as a continuation of Suzuki \cite{suzuki} and Golshani's \cite{golshani} analysis of Shelah cardinals in which they investigated the properties of \textit{witnessing number} of a Shelah cardinal and the behavior of the continuum function in the presence of such large cardinals respectively.

 \section{Shelah cardinals and small forcings}
In this section we show that Shelah cardinals are preserved under (relatively) small forcings in both upward and downward directions. In other words no Shelah cardinal loses its Shelahness in the generic extension produced by a small forcing and such a forcing doesn't add any new Shelah cardinal to the universe as well.

 \begin{theorem}\label{Small forcing and Shelah cardinals}
 Small forcing preserves Shelah cardinals in both upward and downward directions.  i.e.,  If $\kappa$ is a cardinal and $\mathbb{P}$ is a forcing notion with $|\mathbb{P}|<\kappa$ and $G\subseteq \mathbb{P}$ is a generic filter then $\kappa$ is Shelah in $V$ if and only if $\kappa$ is Shelah in $V[G]$.
 \end{theorem}

We need the following lemma from \cite{golshani} stating that for any function $f:\kappa\rightarrow \kappa$ in the generic extension some suitable \textit{dominating functions} exist in the ground model if the forcing notion satisfies some certain properties (including relative smallness).

 \begin{lemma}\label{dominating function lemma}
 If $\kappa$ is a regular cardinal in $V$ and $\mathbb{P}$ satisfies one of the following conditions, then for every function $f:\kappa\rightarrow\kappa$ in a $\PP$-generic extension of $V,$ there exists a function $g:\kappa\rightarrow\kappa$ in $V$ such that $g$ dominates $f$ ,  i.e. ,  $\forall\alpha\in \kappa~~~f(\alpha)<g(\alpha)$ .
 \begin{enumerate}
 \item[(1)] $\mathbb{P}$ is $\kappa$-c.c .
 \item[(2)] $\mathbb{P}$ is $\kappa^{+}$-distributive .
 \end{enumerate}
 \end{lemma}
 We also need the following results of Hamkins \cite{hamkins3}.
 \begin{lemma} \label{Hamkins' lemma for small forcing 1}
  Let $V\subseteq\overline{V}$ be models of $\ZFC$ and let $j:\overline{V}\rightarrow\overline{M}$ be a definable elementary embedding with $crit(j)=\kappa$ and
 $M=\bigcup j[V]$ so that $j|_V$ is an elementary embedding from $V$ to $M$.  If there is a regular cardinal $\delta<\kappa$ in $\overline{V}$ such that $\overline{V}\models\;^{\delta}\overline{M}\subseteq \overline{M}$ and $\delta$-covering and $\delta$-approximation properties hold between $V$ and $\overline{V}$, then:
 \begin{enumerate}
 \item[(1)] $M=\overline{M}\cap V$ .  In particular $V\models\;^{\delta}M\subseteq M$ .
 \item[(2)] For all $A\in V$ ,  $j|_{A}\in V$ .
 \item[(3)] For all ordinal $\lambda$ ,  if $V_{\lambda}\subseteq \overline{M}$ then  $V_{\lambda}\subseteq M$ .
 \item[(4)] For all ordinal $\lambda$ ,  if $\overline{V}\models\;^{\lambda}\overline{M}\subseteq \overline{M}$ then  $V\models\;^{\lambda}M\subseteq M$ .
 \end{enumerate}
 \end{lemma}
 \begin{lemma}\label{Hamkins' lemma for small forcing 2}
  Let $\delta$ be a cardinal and $\mathbb{P}*\lusim{\mathbb{Q}}$ be a forcing iteration such that $\mathbb{P}$ is non-trivial ,  $|\mathbb{P}|\leq \delta$ and $\Vdash_{\mathbb{P}} $``$  \lusim{\mathbb{Q}}$ is $(\delta+1)$-strategically closed'' ,  then the $\delta^+$-covering and $\delta^+$-approximation properties hold between $V$ and $V^{\mathbb{P}*\dot{\mathbb{Q}}}$ .
 \end{lemma}

 Now we are ready to complete the proof of Theorem \ref{Small forcing and Shelah cardinals}.
 \\
 {\bf Proof of Theorem \ref{Small forcing and Shelah cardinals}}
 For the upward direction, suppose that $\k$ is a Shelah cardinal in $V$,  and let $f\in V[G] ,  f:\k \rightarrow \k.$ By Lemma \ref{dominating function lemma} $(1)$,  there exists $g\in V ,  g :  \k \rightarrow \k$ which dominates $f$ .  By assumption ,  there exists $j :  V \rightarrow M \supseteq V_{j(g)(\k)}$ witnessing the Shelahness of $\k$ in $V$ with respect to $g$ .  But then $j$ is easily seen to extend to some $j^* :  V[G] \rightarrow M[G] \supseteq V_{j(g)(\k)}[G],$
  and $V_{j(g)(\k)}[G] \supseteq V_{j^*(f)(\k)}[G]=V_{j^*(f)(\k)}^{V[G]}.$

For the downward direction just note that in the lemma \ref{Hamkins' lemma for small forcing 2} if we take $\mathbb{Q}$ to be trivial forcing then it follows that any forcing notion $\mathbb{P}$ satisfies $|\mathbb{P}|^+$ - approximation and covering properties between $V$ and an extension by $\mathbb{P}$. Particularly if $\kappa$ is a Shelah cardinal in the extension by a forcing notion $\mathbb{P}$ with $|\mathbb{P}|<\kappa$ and $f:\kappa\rightarrow\kappa$ is a function in the ground model, then by inaccessiblity of $\kappa$ in $V$ we have $|\mathbb{P}|^+<\kappa$. As $\mathbb{P}$ also satisfies  $|\mathbb{P}|^+$ - approximation and covering properties between $V$ and $V[G]$, using lemma \ref{Hamkins' lemma for small forcing 1} we can get the required non-trivial elementary embedding $j_{f}$ from $V$ to an inner model $M_{f}$ witnessing Shelahness of $\kappa$ with respect to $f$ in $V$ by restricting a non-trivial elementary embedding $\overline{j_{f}}$ from $V[G]$ to an inner model $\overline{M_{f}}$ witnessing Shelahness of $\kappa$ with respect to $f$ in $V[G]$.\hfill$\Box$
 
 \section{Shelah cardinals and fast function forcing}

In this section we investigate preservation of $(^{\kappa}\kappa \cap V)$-Shelah cardinals, a certain type of Shelah cardinals introduced by Gitik and Shelah \cite{Gitik-Shelah}, under Woodin's fast function forcing. As the name suggests, such a forcing adds a new $\kappa$-sequence of ordinals $<\kappa$ which in some sense is \textit{faster} than any other such $\kappa$-sequence in the ground model.

The definitions are as follows:

\begin{definition}\label{A-shelah cardinal}
If $\kappa$ is an uncountable cardinal and $A$ is a set of functions from $\kappa$ into $\kappa$ then $\k$ is called $A$-Shelah, if for every function $f\in A$ there exists an elementary embedding $j :  V \rightarrow M$ with $crit(j)=\k$ such that $^{\k}M \subseteq M$ and $V_{j(f)(\k)} \subseteq M$.
\end{definition}

\begin{definition} \label{def:fast function forcing}
Let $\k$ be a  cardinal. The fast function forcing $\mathbb{P}_\kappa$ consists of all partial functions $p: \kappa\rightarrow\kappa$ such that:
\begin{enumerate}
\item[(1)] $\forall\gamma\in dom(p)\;\;\; p[\gamma]\subseteq \gamma,$
\item[(2)] $\forall\gamma\leq \kappa$, if $\gamma$ is inaccessible then $|dom(p \upharpoonright_{\gamma})|<\gamma.$
\end{enumerate}
$\PP_\k$ is ordered by reverse inclusion.
\end{definition}
Note that for any $\gamma<\k,$ and any $p\in \PP_\k$ with $\gamma\in \dom(p),$ we can factor $\PP_\k/p$ as $\PP_\k/p \cong (\PP_\gamma/p \upharpoonright_{\gamma})\times (\PP_{[\gamma, \k)}/p \upharpoonright_{[\gamma, \k)}),$ where $\PP_{[\gamma, \k)}=\{p\in \PP_\k: \dom(p) \subseteq [\gamma, \k)\}.$

\begin{theorem}\label{Fast function forcing and Shelah cardinals}
Fast function forcing preserves $(^{\kappa}\kappa \cap V)$ - Shelah cardinals.
\end{theorem}
\begin{proof}
We go through a lifting argument as follows. Let $G$ be $\PP$-generic over $V$, and let $F=\bigcup G $ be the fast function added by $G$. We may note that $V[G]=V[F].$ Let $g\in V$, $g:\kappa\rightarrow\kappa$. Let $j: V \rightarrow M \supseteq V_{j(g)(\k)}$ witness the Shelahness of $\k$ in $V$ with respect to $g.$ We may also suppose that $j$ is an extender embedding, so that
\begin{center}
$M=\{j(h)(s): h\in V$ and $s\in [\delta]^{<\omega}            \},$
\end{center}
where $\delta=|V_{j(g)(\k)}|$. As $\delta$ is definable in $M$ from $\kappa$ we may also assume that:
\begin{center}
$M=\{j(h)(\kappa): h\in V\}$ 
\end{center}

Consider the condition $p=\{\langle \k, \delta \rangle   \}\in \PP^M_{j(\k)}.$ Then
$\PP^M_{j(\k)} \cong \PP_\k \times (\PP^M_{[\k, j(\k))}/p).$ Let $X=\{j(h)(\k, \delta): h\in V   \}.$ Then $X \prec M$ and we have the following commutative diagram
\begin{center}

\begin{align*}
\begin{diagram}
\node{V}
        \arrow{e,t}{j}
        \arrow{s,l}{j_0}
        \node{M}
\\
\node{M_0}
         \arrow{ne,b}{k}
\end{diagram}
\end{align*}
\end{center}
where $k: M_0 \rightarrow M$ is the inverse of the transitive collapse of $X$. Since $\k, \delta \in X,$ there is $\delta_0 < j_0(\k)$ with $k(\delta_0)=\delta.$ Let $p_0=\{\langle \k, \delta_0 \rangle   \}\in j_0(\PP_\k).$ Then $k(p_0)=p$ and $crit(k)>\k.$ We have
\begin{enumerate}
\item $j_0(\PP_\k)/p_0=\PP^{M_0}_{j_0(\k)}/p_0 \cong \PP_\k \times \PP^{M_0}_{[\k, j_0(\k))}/p_0$
\item $V[G]\models$``$\PP^{M_0}_{[\k, j_0(\k))}/p_0$ is $\k^+$-closed'',
\item $V[G]\models$``$|\{A\in M_0[G]: A$ is a maximal antichain in $ \PP^{M_0}_{[\k, j_0(\k))}   \}|\leq \k^+$''.
\end{enumerate}
So there is $H_0\in V[G]$ which is $\PP^{M_0}_{[\k, j_0(\k))}/p_0$-generic over $M_0[G],$ and clearly $j_0[G] \subseteq G\times H_0.$ So we can lift $j_0$ to get
\begin{center}
$j_0^*: V[G] \rightarrow M_0[G\times H_0].$
\end{center}
Now we lift $k$ to $M[G\times H_0].$ Since $crit(k)>\k,$ we can easily lift $k$ to
\begin{center}
$k^*:M_0[G] \rightarrow M[G].$
\end{center}
Let $H=k^*[H_0] \subseteq \PP^M_{[\k, j(\k))}/p.$
\begin{claim}
$H$ is $\PP^M_{[\k, j(\k))}/p$-generic over $M[G].$
\end{claim}
\begin{proof}
Let $D\in M[G]$ be dense open in $\PP^M_{[\k, j(\k))}/p,$ and let $\lusim{D}\in M$ be a name for it. Then $\lusim{D}=k^*(h)(s)$ for some $h\in M_0$ and $s\in [\delta]^{<\omega}.$
We may further suppose that for all $x\in \dom(h), h(x)$ is a $M_0$-name for an open dense subset of $\PP^{M_0}_{[\k, j_0(\k))}/p_0.$ Since the latter forcing is $\delta^+$-distributive in $M_0[G],$ the set $D_0=\bigcap_{x}h(x)[G]$ is dense in it, and hence $H_0\cap D_0\neq \emptyset,$ which implies $H\cap D\neq \emptyset.$
\end{proof}
Thus we can lift $k^*$ further to
\begin{center}
$k^{**}:M_0[G\times H_0] \rightarrow M[G\times H].$
\end{center}
It follows that we can lift $j$ to $j^*: V[G] \rightarrow M[G\times H],$
and clearly $(V[G])_{j(g)(\k)}\subseteq M[G\times H].$ 

Hence $j^*$ witnesses the Shelahnes of $\k$ in $V[G]$ with respect to $g$, and the theorem follows immediately.
\end{proof}
 
 \section{Shelah cardinals and Laver diamond principle}

First we need to introduce the essential notion of the \textit{witnessing number} associated to a Shelah cardinal which is a central tool for analyzing Shelah cardinals. It has been introduced by Suzuki \cite{suzuki}.

 \begin{definition}
For a Shelah cardinal $\kappa$ let $wt(\kappa)$ be the least ordinal $\lambda$ such that for any function $f:\kappa\rightarrow\kappa$ there exists an extender $E\in V_{\lambda}$ which witnesses the Shelahness of $\kappa$ with respect to $f$. We call $wt(\kappa)$ the witnessing
number of $\kappa$.
 \end{definition}

Inspired by similar definition of the Laver Diamond Principle for other large cardinals in \cite{hamkins-diamond}, the Laver Diamond Principle for a Shelah cardinal is defined as follows:

\begin{definition}
$\diamond_{\kappa}^{Shelah}$ is the assertion: there exists $l:\kappa\rightarrow V_{\kappa}$ such that for any $f: \kappa\rightarrow\kappa$ and any $\alpha < wt(\kappa)$ and $A\in V_{\alpha}$, there are $g>f$ and $j:V\rightarrow M$ witnessing the Shelahness of $\kappa$ with respect to $g$ such that $j(g)(\kappa)>\alpha$ and $j(l)(\kappa)=A$.
\end{definition}

The use of $\alpha$ and $g$ in the above definition is because we need a witnessing embedding which includes $V_{\alpha}$. It will be guaranteed by the condition $j(g)(\kappa)>\alpha$ which might not be the case for arbitrary $f$. Also the condition $g>f$ guarantees that the embedding witnessing Shelahness with respect to $g$ also witnesses Shelahness with respect to $f$.    

\begin{theorem}\label{shelah and diamond}
If $\kappa$ is a Shelah cardinal then $\diamond_{\kappa}^{Shelah}$ holds.
\end{theorem}

\begin{proof}
Fix a well-ordering $\lhd$ of $V_{wt(\kappa)}$. We construct the function $l:\kappa\rightarrow V_{\kappa}$ by transfinite recursion. If $l\upharpoonright _{\gamma}$ has been defined then we define $l(\gamma)$ as follows. Let $\lambda $ be the least ordinal such that there are   $a\in V_{\lambda}$ and  $f: \gamma\rightarrow \gamma$, such that for all
$h:\gamma\rightarrow \gamma$ with $h>f$ and all $j:V\rightarrow M$ with $crit(j)=\gamma$ and $V_{j(h)(\gamma)}\subseteq M$, if $j(h)(\gamma)>\lambda$ then we have $j(l)(\gamma)\neq a$. If there is such a $\lambda$ we define $l(\gamma)=a$ for the $\lhd$-minimal $a$ with the stated property.

Now we claim that $l$ is the required $\diamond_{\kappa}^{Shelah}$ Laver function. If not, let $\theta <wt(\kappa)$ be the least ordinal such that there is  $f: \kappa\rightarrow\kappa$ and some $a\in V_{\theta}$ with the property that for all $g>f$ and all elementary embeddings $j:V\rightarrow M$ witnessing the Shelahness of $\kappa$ with respect to $g$ if $j(g)(\kappa)>\theta$ then $j(l)(\kappa)\neq a$. 

Fix a suitable function $g:\kappa\rightarrow\kappa$, $g>f$, $j(g)(\kappa)>\theta$, and its corresponding Shelahness embedding $j:V\rightarrow M$. By the choice of $g$, as we have $V_{j(g)(\kappa)}\subseteq M$ and $j(g)(\kappa)>\theta$, then $M$ and $V$ are agree on $V_{\theta +1}$. Also $M$ contains enough Shelahness extenders and will be agree with $V$ on the fact that $\theta$ is the least ordinal with the stated property, namely $a$ is not anticipated by $l=j(l)\upharpoonright \kappa$ (due to the critical point of the elementary embedding that is $\kappa$). Thus by applying the recursive definition of $l$ in M we will have $j(l)(\kappa)=a'$ for some $a'$ in $M$ that is $j(\lhd)$-minimal. 

Thus $M$ thinks ``$\theta $ is the least ordinal such that there are   $a'\in V_{\theta}$ and  $f: \k\rightarrow \k$, such that for all
$h:\k\rightarrow \k$ with $h>f$ and all extenders $E$ if $j_E(h)(\k)>\theta$ then we have $j_E(l)(\k)\neq a'$''.

But then there should be no such extender in $V$, which is not possible as $j$ itself gives such an extender, and we get a contradiction.
\end{proof}

 \section{Shelah cardinals and closed forcings}
In this section we show that Shelah cardinal $\kappa$ can be made indestructible under $\leq\kappa$-closed forcings of size $<wt(\kappa)$. The following proposition from \cite{golshani} shows that the restriction on size of the forcing is essential and so the follow up theorem is sharp.  

\begin{proposition}
If $\kappa$ is Shelah and $\lambda < wt(\kappa)$ is a regular cardinal then there is a $\lambda$-closed forcing $\mathbb{P}$ of size $wt(\kappa)$ such that $\kappa$ is no longer Shelah in $V^{\mathbb{P}}$.
\end{proposition} 

In order to prove the main result first note to the following lemma from \cite{golshani} which states that without loss of generality one may assume $GCH$ in a model of Shelah cardinals. 

\begin{lemma}\label{shelah and gch}
The canonical forcing of $GCH$ preserves all Shelah cardinals. 
\end{lemma}

Now we present our main theorem. For the proof we generally follow the method that is used in \cite{johnstone} for the case of strong cardinals. 

\begin{theorem}\label{shelah and closed}
If $\kappa$ is Shelah and $GCH$ holds then there is a set forcing extension in which the Shelahness of $\kappa$ becomes indestructible by any $\leq\kappa$-directed closed forcing of size $<wt(\kappa)$.
\end{theorem}
\begin{proof}
We define an Easton support iterated forcing of length $\kappa$ using the Laver function $l$ for Shelah cardinals as obtained in theorem \ref{shelah and diamond}. If $\mathbb{P}_{\gamma}$ for $\gamma<\kappa$ is already defined, and if $l(\gamma)$ is a $\mathbb{P}_{\gamma}$-name for a $\leq\gamma$-closed poset in $V^{\mathbb{P}_{\gamma}}$ then we define the forcing at stage $\gamma$, namely $\mathbb{Q}_{\gamma}$, to be this poset. Otherwise we define $\mathbb{Q}_{\gamma}$ to be trivial forcing. 

Now we prove that the forcing $\mathbb{P}$ is as required. Let $G\subseteq \mathbb{P}$ be a $V$-generic filter and  $\mathbb{Q}$ is a $\leq\kappa$-directed closed forcing of size $<wt(\kappa)$ in $V[G]$. We need to prove that $\kappa$ remains Shelah in $V[G][g]$, where $g\subseteq \mathbb{Q}$ is a $V[G]$-generic filter. 

Fix a function $f:\kappa\rightarrow\kappa$ in $V[G*g]$. Note that by $\kappa^+$-directed closure of $\mathbb{Q}$, it adds no $\kappa$-sequences to $V[G]$ so $f\in V[G]$. Also by $\kappa$-cc property of $\mathbb{P}$ and lemma \ref{dominating function lemma} there is a dominating function $g:\kappa\rightarrow\kappa$ for $f$ in $V$. By the size of the forcing $\mathbb{Q}$ we may assume that $\dot{\mathbb{Q}}\in V_{\theta}$ for some $\theta< wt(\kappa)$. As $l$ is a Laver function for Shelah cardinals there will be a function $h:\kappa\rightarrow\kappa$ with $h>g$ and an elementary embedding $j:V\rightarrow M$ witnessing Shelahness of $\kappa$ with respect to $h$ such that $j(h)(\kappa)>\theta$ and $j(l)(\kappa)=\dot{\mathbb{Q}}$.

Without loss of generality we may assume that $M=\{j(i)(s)~|~i:V_{\kappa}\rightarrow V, i\in V, s\in V_{\theta}\}$. It is easy to check that $j(\mathbb{P})$ and $\mathbb{P}$ are the same in the first $\kappa$- stages because $\mathbb{P}$ is defined relative to $l$ and $M[G]$ agrees that $\dot{\mathbb{Q}}$ is a name for a $\leq\kappa$-closed forcing notion. At stage $\kappa$ itself $j(\mathbb{P})$ is $\dot{\mathbb{Q}}$. So we have $j(\mathbb{P})=\mathbb{P}*\dot{\mathbb{Q}}*\mathbb{P}_{tail}$. Also we may assume that $\forall\gamma \in dom(l)~~~l[\gamma]\subseteq V_{\gamma}$, and $\mathbb{P}_{tail}$ is $\leq\beth_{\theta}$-closed in $M[G][g]$ and $\exists l'\in$ $^\kappa \kappa$ $\theta = j(l')(\kappa)$. 

Now consider the elementary embedding $j:V\rightarrow M$ witnessing Shelahness of $\kappa$ with respect to $h$ in $V$. In $V[G*g]$ we lift $j$ in two steps to obtain an elementary embedding $j:V[G]\rightarrow M[j(G)]$ and then $j: V[G*g]\rightarrow M[j(G)*j(g)]$.

\textbf{Step 1}: Lifting $j:V\rightarrow M$ to $j:V[G]\rightarrow M[j(G)]$ in $V[G*g]$.  

We need to build an $M$-generic filter $j(G)\subseteq j(\mathbb{P})$ in $V[G*g]$. As $j(\mathbb{P})=\mathbb{P}*\dot{\mathbb{Q}}*\mathbb{P}_{tail}$ and $G*g\subseteq \mathbb{P}*\dot{\mathbb{Q}}$ is $V$-generic and so $M$-generic, we only need to find suitable $M[G*g]$-generic filter $G_{tail}\subseteq \mathbb{P}_{tail}$ in $V[G*g]$ to form the required $j(G)$ as $G*g*G_{tail}$. Due to the lack of enough closure in $M[G*g]$ we work with a nice elementary substructure of it as follows:

Let $X:=\{j(f)(\kappa)~|~f:\kappa\rightarrow V~,~f\in V\}$. $X$ is an elementary substructure of $M$ and contains $\kappa$ and all what exists in $ran(j)$. Thus we have $\theta, V_{\theta}, \mathbb{P}, \dot{\mathbb{Q}}, j(\mathbb{P}), j(l), \mathbb{P}_{tail}\in X$. Consequently $\mathbb{Q}\in X[G]$ and by Tarski criterion $X[G]\prec M[G]$ and $X[G*g]\prec M[G*g]$.

It is straightforward to check that for every dense open subset $D$ of $\mathbb{P}_{tail}$ in $M[G*g]$ there is a dense subset $\overline{D}$ of $\mathbb{P}_{tail}$ in $X[G*g]$ such that $\overline{D}\subseteq D$. Based on this fact we only need to find suitable $X[G*g]$-generic filter $G_{tail}\subseteq \mathbb{P}_{tail}$ in $V[G*g]$ which will be $M[G*g]$-generic as well.

We use a diagonalization argument in order to find $X[G*g]$-generic filter $G_{tail}$. It is easy to check that $^\kappa X\subseteq X$ in $V$. By $\kappa$-cc property of $\mathbb{P}\subseteq X$ we have $^\kappa X[G]\prec X[G]$ in $V[G]$. Also from $\leq \kappa$-distributivity of $\mathbb{Q}$ it follows that  $^\kappa X[G*g]\prec X[G*g]$ in $V[G*g]$. Futhurmore $X[G*g]\prec M[G*g]$ implies that $\mathbb{P}_{tail}$ is $\leq\kappa$-closed in $X[G*g]$.

Now we need to calculate the number of maximal antichains of $\mathbb{P}_{tail}$ that exist in $X[G*g]$ from $V[G*g]$ perspective. In order to do so note that $|\mathbb{P}_{tail}|=j(\kappa)$. The forcing $\mathbb{P}*\dot{\mathbb{Q}}$ doesn't change the size of $P(j(\kappa))$ so all what we we need is to calculate the size of $P(j(\kappa))\cap X$ in $V$, that is $2^\kappa = (\kappa^+)^{V}$ by $GCH$ in $V$. Thus the number of maximal antichains of $\mathbb{P}_{tail}$ in $X[G*g]$ from $V[G*g]$ perspective would be $(\kappa^+)^{V}\leq (\kappa^+)^{V[G*g]}$. Thus there is an enumeration of these maximal antichains by a $\kappa^+$-sequence in $V[G*g]$. 

Finally by diagonalization we construct a descending $\kappa^{+}$-sequence of conditions in $X[G*g]\cap \mathbb{P}_{tail}$ which meets every antichain of $\mathbb{P}_{tail}$ that exists in $X[G*g]$. Now take $G_{tail}\subseteq \mathbb{P}_{tail}$ to be the filter generated by such a sequence. It will be $X[G*g]$-generic and so $M[G*g]$-generic and the lifting will be possible.

\textbf{Step 2}: Lifting $j:V[G]\rightarrow M[j(G)]$ to $j: V[G*g]\rightarrow M[j(G)*j(g)]$ in $V[G*g]$. 

Denote the filter generated by $j[g]$ as $<j[g]>\subseteq j(\mathbb{Q})$ which meets every dense subset $D\in ran(j)$ of $j(\mathbb{Q})$. Again it is straightforward to check that for every dense open subset of $j(\mathbb{Q})$ such as $D\in M[j(G)]$ there is a $\overline{D}\in ran(j)$ such that $\overline{D}$ is a dense subset of $j(\mathbb{Q})$ and $\overline{D}\subseteq D$. Thus $<j[g]>$ is a $M[j(G)]$-generic filter. Consequently if we consider $j(g)=<j[g]>$ then in $V[G*g]$ it would be possible to lift $j:V[G]\rightarrow M[j(G)]$ to $j: V[G*g]\rightarrow M[j(G)*j(g)]$. 

Finally we have $V_{j(f)(\kappa)}^{V[G*g]}=V_{j(f)(\kappa)}[G*g]\subseteq V_{j(h)(\kappa)}[G*g]\subseteq M[G*g]\subseteq M[j(G)*j(g)]$. Thus the lifted embedding witnesses Shelahness of $\kappa$ with respect to $f$ in $V[G*g]$.

\end{proof}

%%%
%%% Acknowledgments
%%%
\section*{\bf Acknowledgments}
The authors would like to thank Mohammad Golshani for careful reading of the primary draft and the anonymous referee for making useful comments which led to various improvements in the text. We also would like to especially thank the journal editor, Ali Enayat, for his sincere efforts during the communications between the authors and referee.

%------------------------------------------------------------------------------------%

\end{document}